\documentclass [reqno,10pt,oneside]{amsart}

\usepackage{amsthm}
\usepackage{amscd}
\usepackage{amsmath,amssymb}
\usepackage{algorithm}
\usepackage[noend]{algorithmic}

\theoremstyle {plain}
\newtheorem {thm}{Theorem}[section]

\newtheorem {lem}[thm]{Lemma}

\theoremstyle {definition}
\newtheorem {defn}[thm]{Definition}
\newtheorem {alg}[thm]{Algorithm}
\theoremstyle {remark}

\newtheorem {exmp}[thm]{Example}

\DeclareMathOperator{\Ann}{Ann}

\DeclareMathOperator{\height}{ht}
\DeclareMathOperator{\LM}{LM}
\DeclareMathOperator{\LT}{LT}
\DeclareMathOperator{\LC}{LC}
\DeclareMathOperator{\lcm}{lcm}

\newcommand{\F}{{\mathbb F}}

\newcommand{\Q}{{\mathbb Q}}

\newcommand{\Z}{{\mathbb Z}}

\newcommand{\gen}[1]{\left\langle #1 \right\rangle}

\newcommand{\singular}{{\sc Singular }}

\begin{document}

\bibliographystyle{alpha}

\title{An Algorithm to Compute a Primary Decomposition of Modules in Polynomial Rings over the Integers}

\author{Nazeran Idrees}
\address{Nazeran Idrees\\ Department of Mathematics\\
GC University\\
Kotwali Road,Jinnah Town, Faisalabad 38000,Punjab\\
Faisalabad Pakistan\\}
\email{nazeranjawwad@gmail.com}

\author{Gerhard Pfister}
\address{Gerhard Pfister\\ Department of Mathematics\\ University of Kaiserslautern\\
Erwin-Schr\"odinger-Str.\\ 67663 Kaiserslautern\\ Germany}
\email{pfister@mathematik.uni-kl.de}
\urladdr{http://www.mathematik.uni-kl.de/$\sim$pfister}

\author{Afshan Sadiq}
\address{Afshan Sadiq, \\
Department of Mathematics, \\
Jazan University, \\
P.O. Box 114, Jazan, Saudia Arabia.}
\email{afshansadiq6@gmail.com}


\date{\today}

\maketitle

\begin{abstract}
We present an algorithm to compute the primary decomposition of a submodule $\mathcal{N}$
of the free module $\Z[x_1, \ldots, x_n]^m$. For this purpose we use
algorithms for primary decomposition of ideals in the polynomial ring over the integers.
The idea is to compute first the minimal associated primes of $\mathcal{N}$, i.e. the minimal associated primes of the ideal
$\Ann(\Z[x_1, \ldots, x_n]^m /\mathcal{N})$ in $\Z[x_1,\ldots,x_n]$
and then compute the primary components using  pseudo-primary decomposition and  extraction, following the ideas of Shimoyama-Yokoyama. The algorithms are implemented in {\sc Singular}.
\end{abstract}

\section{Introduction}
Algorithms for primary decomposition in $\Z[x_1,\ldots,x_n]^m$ have been developed by
Seidenberg (cf. \cite{Se}) and Ayoub (cf. \cite{A}) and Gianni, Trager and Zacharias (cf. \cite{GTZ}).
The method of Gianni, Trager and Zacharias have been generalized by Rutman (\cite{R}) to a submodules of a free module.
In our paper we present a slightly different approach using pseudo--primary
decomposition, and the extraction of the primary components.
We use the computation of minimal associated primes of ideals in $\Z[x_1, \ldots, x_n]$ (cf. \cite{PSS}).
\\ Let us recall the primary decomposition for ideals in $\Z[x], \,\, x=(x_1, \ldots,x_n)$, since the ideas for submodules of $\Z[x]^m$ are similar.
The idea to compute the minimal associated prime ideals of an ideal $I\subseteq \Z[x]$ is the following.
We compute a Gr\"obner basis $G$ of $I$ (cf. Definition \ref{grobner}). $G\cap \Z$ generates $I\cap \Z$. If $I\cap \Z = \gen{a}$ and $a = p_1^{v_1} \cdot \ldots \cdot p_s^{v_s}$ the prime decomposition then we compute for all $i$ the minimal associated primes of $I\F_{p_i}[x]$, defined by the canonical map $\pi_i : \Z[x] \longrightarrow \F_{p_i}$. If $\overline{P}$ is a minimal associated prime of $I\F_{p_i}[x]$ then $\pi_i^{-1}(\overline{P})$ is a minimal associated prime of $I$. We obtain all minimal associated primes of $I$ in this way. If $I\cap \Z = \gen{0}$ then we consider the ideal $I\Q[x]$ and compute its minimal associated primes. If $\overline{P}\supset I\Q[x]$ is a minimal associated prime then $\overline{P}\cap \Z[x]$ is a minimal associated prime of $I$. Using the leading coefficients of a Gr\"obner basis of $I\Q[x]$ we find $h\in \Z$ such that $I\Q[x] \cap \Z[x] = I : h$ and $I = (I:h)\cap\langle I, h \rangle$.
The minimal associated primes of $\langle I, h\rangle$ can be computed as described above.
\\ If we know the minimal associated primes $B = \{P_1, \ldots, P_r\}$ of $I$ we can find knowing Gr\"obner bases of $P_i$ a set of separators $S = \{s_1, \ldots, s_r\}$ wiith the property $s_i \notin P_i$ and $s_i \in P_j$ for all $j\neq i$. Using the separators we find pseudo--primary ideals $Q_i = I:s_i^\infty=I:s_i^{k_i}$, i.e. the radical of $Q_i$ is prime (cf. Definition \ref{defpseudo}), $\sqrt{Q_i}=P_i$. We have
$$I=Q_1 \cap \ldots \cap Q_r \cap \langle I, s_1^{k_1}, \ldots s_r^{k_r}\rangle.$$
\\To find a primary decomposition we have to continue inductively with $\langle I, s_1^{k_1}, \ldots s_r^{k_r}\rangle$ and find from $Q_i$ the primary ideal of $I$ with associated prime $P_i$. This is given by so--called extraction lemma (the version for modules is Lemma \ref{lemExtraction}. If $Q_i = \overline{Q}_i\cap J$, $\overline{Q}_i$ primary and $\sqrt{\overline{Q}_i}=P_i$, $\height(J) > \height(P_i)$, then we can extract the $\overline{Q}_i$ using Gr\"obner bases with respect to special orderings.
\section{Gr\"obner basis for Modules}
Let $A$ be a principal ideal domain. Let $\mathcal{M}=A[x]^m$, $m>0$ be the free module over the polynomial ring over $A$, $x=\{x_1, \ldots, x_n\}$ and $\textbf{\textsf{e}}_1, \ldots, \textbf{\textsf{e}}_m$ the canonical basis of $\mathcal{M}$.
In this section we give basic results about Gr\"obner bases for modules in $\mathcal{M}$.

A monomial ordering $>$ is a total ordering on the set of monomials $\operatorname{Mon}_n\!=\!\{x^\alpha \!=x_1^{\alpha_1}\cdot \ldots \cdot x_n^{\alpha_n} | \alpha =(\alpha_1,\ldots ,\alpha_n) \in \mathbb{N}^n\}$ in $n$ variables satisfying $$x^\alpha > x^\beta \Longrightarrow x^\gamma x^\alpha > x^\gamma x^\beta$$
for all $\alpha, \beta, \gamma  \in \mathbb{N}^n$. Further more $>$ must be a well--ordering. We extend the notion of monomial orderings to the free module $\mathcal{M}$. We call $x^\alpha \textbf{\textsf{e}}_i=(0, \ldots, x^\alpha, \ldots, 0)\in \mathcal{M}$.
\begin{defn}
Let $>$ be a monomial ordering on $A[x]$. A (module) monomial ordering or a module ordering on $\mathcal{M}$ is a total ordering $>_m$ on the set of monomials $\{x^\alpha \textbf{\textsf{e}}_i | \alpha \in \mathbb{N}^n, \,\, i=1 \ldots m\}$, which is compatible with the $A[x]$-module structure including the ordering $>$, that is, satisfying
\begin{itemize}
\item[1.] $x^\alpha \textbf{\textsf{e}}_i >_m x^\beta \textbf{\textsf{e}}_j \Longrightarrow x^{\alpha + \gamma} \textbf{\textsf{e}}_i >_m x^{\beta + \gamma} \textbf{\textsf{e}}_j$,
\item[2.] $x^\alpha > x^\beta \Longrightarrow x^\alpha \textbf{\textsf{e}}_i >_m x^\beta \textbf{\textsf{e}}_i$,
\end{itemize}
for all $\alpha, \beta, \gamma \in \mathbb{N}^n, i, j = 1, \ldots, r$.
\end{defn}
Two module ordering are of particular interest:
$$x^\alpha \textbf{\textsf{e}}_i > x^\beta \textbf{\textsf{e}}_j : \Longleftrightarrow i < j \,\, \text{or} \,\,(i = j \,\, \text{and} \,\, x^\alpha > x^\beta),$$
giving priority to the component, denoted by $(c, >)$, and
$$x^\alpha \textbf{\textsf{e}}_i > x^\beta \textbf{\textsf{e}}_j : \Longleftrightarrow x^\alpha > x^\beta \,\, \text{or} \,\,(x^\alpha = x^\beta \,\, \text{and} \,\, i < j),$$
giving priority to the monomial in $A[x]$, denoted by $(>, c)$.

Now fix a module ordering $>_m$ and denote it also with $>$. Since any vector $f\in \mathcal{M}\setminus \{0\}$ can be written uniquely as $$f=cx^\alpha \textbf{\textsf{e}}_i + f^*$$ with $c\in A\setminus \{0\}$ and $x^\alpha \textbf{\textsf{e}}_i > x^{\alpha^*} \textbf{\textsf{e}}_j$ for every non-zero term $c^* x^{\alpha^*} \textbf{\textsf{e}}_j$ of $f$ we can define as $\LM(f) := x^\alpha \textbf{\textsf{e}}_i$, $\LC(f) := c$, $\LT(f) := c x^\alpha \textbf{\textsf{e}}_i$ and call it the leading monomial, leading coefficient and leading term \footnote{If we want to hint that we consider $f$ in the ring $A[x]$ we write for the leading term $\LT_{A[x]}(f)$.}, respectively, of $f$. Moreover, for $G\subset \mathcal{M}$ we call $$L_>(G) := L(G) := \langle \LT(g) \,\, | \,\, g\in G\setminus \{0\}\rangle_{A[x]} \subset \mathcal{M}$$ the leading submodule of $\langle G \rangle$. In particular, if $\mathcal{N}\subset \mathcal{M}$ is a submodule, then $L_>(\mathcal{N})=L(\mathcal{N})$ is called the leading module of $\mathcal{N}$.
\begin{defn}\label{grobner}
Let $\mathcal{N}\subset \mathcal{M}$ be a submodule. A finite set $G\subset \mathcal{N}$ is called a Gr\"obner basis of $N$ if and only if $L(G) = L(\mathcal{N})$. $G$ is called a strong Gr\"obner bases of $\mathcal{N}$, if for any $f\in \mathcal{N}\backslash \{0\}$ there exists $g\in G$ such that $\LT(g)$ divides $\LT(f)$.
\end{defn}
Strong Gr\"obner bases always exist over $A[x]$ (cf. \cite{WP}). If $A$ is not a principal ideal domain then this is not true in general.
\\The concept of a normal form with respect to a given system of elements in $\mathcal{M}$ is the basis of the theory of Gr\"obner bases. We explain this by giving an algorithm.
\\ For terms $ax^\alpha \textbf{\textsf{e}}_i$ and $b x^\beta \textbf{\textsf{e}}_k$, $a, b \in A$, we say $a x^\alpha \textbf{\textsf{e}}_i$ divides $b x^\beta \textbf{\textsf{e}}_k$ and write $a x^\alpha \textbf{\textsf{e}}_i \mid b x^\beta \textbf{\textsf{e}}_k$ if and only if $i=k$, $a \mid b$ and $x^\alpha \mid x^\beta$.
\begin{alg}
\textsc{NF($f | S$)} \label{algPrimdecZ}
\begin{algorithmic}
\REQUIRE $S = \{f_1, \ldots, f_m\}\subseteq \mathcal{M}$, $f\in \mathcal{M}$.
\ENSURE  $r\in \mathcal{M}$ the normal form $NF(f|S)$ with the following properties: $r=0$ or no monomial of $r$ is divisible by a leading monomial of an element of $S$. There exist a representation (standard representation) $f=\sum_{i=1}^m \xi_i f_i + r$, $\xi_i \in A[x]$ such that $\LM(f) \geq \LM(\xi_i f_i)$.
\vspace{0.1cm}
\IF{$f=0$}
\RETURN$f$;
\ENDIF
\STATE $T:= \{g \in S\,\, , \,\, \LT(g) \mid \LT(f)\}$;
\WHILE{$(T \neq \emptyset \,\, \text{and} \,\, f\neq 0)$}
\STATE choose $g\in T,\,\, \LT(f) = h \LT(g)$;
\STATE $f := f - hg$;
\STATE $T:= \{g\in S \,\, , \LT(g) \mid \LT(f)\}$;
\ENDWHILE
\IF{$f=0$}
\RETURN$f$;
\ENDIF
\RETURN$(\LT(f) + NF(f - \LT(f) | S))$;
\end{algorithmic}
\end{alg}

\section{Primary Decomposition for Modules} \label{secBasDefRes}
 First we introduce the notion of a pseudo--primary submodule and show how to decompose a module as intersection of pseudo--primary modules. Then we show how to extract the primary decomposition from a pseudo--primary module.

\begin{defn}\label{defpseudo}
An ideal $I$ of $\Z[x]$ is called a pseudo primary ideal if $\sqrt{I}$ is a prime ideal.
 A submodule $\mathcal{N}\subset \mathcal{M}$ is called a pseudo primary resp. primary submodule of $\mathcal{M}$ if $\Ann(\mathcal{M}/\mathcal{N})$ is a
pseudo primary resp. primary ideal of $\Z[x]$.
\end{defn}

\begin{defn}
Let $\mathcal{N}$ be a submodule of $\mathcal{M}$ and let
$B=\{P_1,P_2,\ldots,P_r\}$ be the set of minimal associated primes. A set $S=\{s_1, \ldots, s_r\}$ is called a system of separators for $B$ if $s_i\notin P_i$ and $s_i\in P_j$ for $j\neq i$.
 \end{defn}

\begin{lem}[Pseudo--Primary Decomposition] \label{lemIntPrimaryComp}
Let $\mathcal{N}\subseteq\mathcal{M}$ be submodule, $B=\{P_1, \ldots, P_r\}$ be the set of minimal associated primes and $S=\{s_1, \ldots, s_r\}$ a system of separators for $B$. Let $$Q_i = N: s_i^{\infty}= N: s_i^{k_i}$$ then $Q_i$ is a pseudo--primary submodule and
$$\mathcal{N}=Q_1\cap \ldots \cap Q_r\cap \langle \mathcal{N}+ \langle s_1^{k_1}, \ldots , s_r^{k_r}\rangle \mathcal{M} \rangle.$$
\end{lem}

\begin{proof}
$\Z[x]_{s_i}\mathcal{N}\cap \mathcal{M}$ is pseudo--primary submodule with minimal associated prime $P_i$.
We obtain this module as a quotient $\mathcal{N}:s^{\infty}_i=\Z[x]_{s_i}\mathcal{N}\cap \mathcal{M}$. This proves that $Q_i$ is pseudo--primary

As in the case of ideals we have $\mathcal{N}=Q_1 \cap (\mathcal{N} + s_1^{k_1}\mathcal{M})$ (cf. \cite{GP}). Assume we have already \\$\mathcal{N}=Q_1 \cap \ldots \cap Q_{t-1} \cap (\mathcal{N}+\langle s_1^{k_1}, \ldots, s_{t-1}^{k_{t-1}}\rangle \mathcal{M})$, $t\leq r$ then
\\$\mathcal{N}=Q_1 \cap \ldots \cap Q_{t} \cap (\mathcal{N}+\langle s_1^{k_1}, \ldots, s_{t}^{k_t}\rangle \mathcal{M})$ since
\\$(\mathcal{N}+\langle s_1^{k_1}, \ldots, s_{t-1}^{k_{t-1}}\rangle \mathcal{M}): s_t^{k_t}= \mathcal{N}: s_k^{k_t}$.
\\The last equality hold since $(\mathcal{N}:s_i^{\infty}):s_j^{\infty}= \mathcal{M}$ if $i\neq j$.
\end{proof}

\begin{defn}
Let $I\subset \Z[x_1, \ldots, x_n]$ be a prime ideal. Let $I\cap \Z=\gen p$ and $\F_p$ the prime field of characteristic $p$.  A subset $$u\subset x = \{x_1, \ldots, x_n\}$$ is called an
independent set (with respect to $I$) if $I\F_p[x]\cap \F_p[u]=\gen 0$. An
independent set $u\subset x$ (with respect to $I$) is called a maximal if the number of elements is maximal\footnote{The number of elements in a maximal independent set $u$ for $I$ is the dimension of $\F_p[x] / I \F_p[x]$.}.
\end{defn}

\begin{lem}[Extraction Lemma] \label{lemExtraction}
Let $\mathcal{N}=Q\cap J$ be a pseudo--primary submodule of $\mathcal{M}$ with $\sqrt{\Ann(\mathcal{M}/Q)}=P$ and $Q$ be $P$--primary with
$\height(\Ann(\mathcal{M}/Q))$\\$<\height(\Ann(\mathcal{M}/J))$. Let $u\subset x$ be a maximal independent set for $P$. Let $P\cap \Z=\gen p$ and define $q:=\left\{\begin{array}{ll}
    1 \,\,\text{if}\,\, p=0 \\
    p \,\,\text{if}\,\, p>0
  \end{array}
\right.$. Let $A:= \Z[x]_{\gen{p}}$\footnote{We are treating the two cases $p=0$ or $p\neq 0$, together. If $p=0$ then $A=\F_p(u)=\Q(u)$ is a field.
If $p\neq 0$ then $A$ is a discrete valuation ring with residue field $\F_p(u)$. Especially we have in both cases the existence of strong Gr\"obner basis over $A[x\smallsetminus u].$}. Then the following holds:
\begin{enumerate}
\item[1.] $\mathcal{N} A[x \smallsetminus u]\cap \mathcal{M}=Q.$
\item[2.] Let $G$ be a strong Gr\"obner basis of $\mathcal{N}$ w.r.t. a block
           ordering satisfying $(x \smallsetminus u) \textbf{\textsf{e}}_i\gg u \textbf{\textsf{e}}_j$. Then $G$ is a strong Gr\"obner
           basis of $\mathcal{N}A[x\smallsetminus u]$ w.r.t. the induced ordering
           for the variables $x\smallsetminus u$.
\item[3.] Let $G$ be a strong Gr\"obner basis of $\mathcal{N}$ w.r.t. a block ordering satisfying
           $(x \smallsetminus u)\textbf{\textsf{e}}_i \gg u \textbf{\textsf{e}}_j$, $\LT_{A[x\smallsetminus u]}(g_i) =
           q^{\nu_i} a_i(x\smallsetminus u)^{\beta_i}\textbf{\textsf{e}}_j$ with $a_i\in \Z[u]\smallsetminus \gen p$
           for $i = 1,\ldots,k$, and $h= \lcm(a_1,\ldots,a_k)$. Then
           $\mathcal{N}A[x\smallsetminus u]\cap \mathcal{M}=\mathcal{N}:h^\infty$.
\end{enumerate}
\end{lem}
Before proving the lemma let us illustrate it by an example.
\\Consider the module \\$\mathcal{N}=\langle\left(
     \begin{array}{c}
       0 \\
       0 \\
       xy^2-x^2-xy \\
     \end{array}
   \right), \left(
              \begin{array}{c}
                0 \\
                y \\
                x \\
              \end{array}
            \right),\left(
                     \begin{array}{c}
                       0 \\
                       x \\
                       2xy-x \\
                     \end{array}
                   \right),\left(
                            \begin{array}{c}
                              x \\
                              0 \\
                              -xy \\
                            \end{array}
                          \right),\left(
                            \begin{array}{c}
                              0 \\
                              0 \\
                              18x \\
                            \end{array}
                          \right)\rangle, $ \\[1.5ex]
  \noindent with $\mathcal{N} \subseteq \Z[x.y]^3 = \mathcal{M}$.
\\Let us compute $\Ann(\mathcal{M}/\mathcal{N}) = \mathcal{N}:\mathcal{M}$.
\begin{verbatim}
ring R=integer,(x,y),lp;
module N=[0,0,xy2-x2-xy],[0,y,x],[0,x,2xy-x],[x,0,-xy],[0,0,18x];
ideal I=quotient(N,freemodule(nrows(N)));
I;
I[1]=18x
I[2]=xy2
I[3]=x2-2xy2+xy
\end{verbatim}
We can see that $P=\gen x$ is the only minimal prime associated to $I$. In this case obviously $u = \{y\}$ is the maximal independent set. We use the following ordering on $\Z[x,y]^3$:
\\$x^i y^ j \textbf{\textsf{e}}_k > x^l y^m \textbf{\textsf{e}}_j\,\, \text{if and only if} \,\, i > l \,\, \text{or}\,\, i=l,\,\, j>m\,\, \text{or}\,\, i=l,\,\, j=m \,\, \text{and}\,\, k<j$.
\\ Let us compute Gr\"{o}bner basis with respect to this ordering.
\begin{verbatim}
std(N);
_[1]=18y*gen(2)
_[2]=y3*gen(2)
_[3]=x*gen(1)+y2*gen(2)
_[4]=x*gen(2)-2y2*gen(2)+y*gen(2)
_[5]=x*gen(3)+y*gen(2)
\end{verbatim}
The Gr\"{o}bner basis is $G=\{\left(
     \begin{array}{c}
       0 \\
       18y \\
       0 \\
     \end{array}
   \right), \left(
              \begin{array}{c}
                0 \\
                y^3 \\
                0 \\
              \end{array}
            \right),\left(
                     \begin{array}{c}
                       x \\
                       y^2 \\
                       0 \\
                     \end{array}
                   \right),\left(
                            \begin{array}{c}
                              0 \\
                              x-2y^2+y \\
                              0 \\
                            \end{array}
                          \right),\left(
                            \begin{array}{c}
                              0 \\
                              y \\
                              x \\
                            \end{array}
                          \right)\}$ \\[1.5ex]
  \noindent Since $P\cap \Z = 0$ we have $\Z[u]_{\gen p}= \Q(y)$. $\mathcal{N}\Q(y)[x]$ is generated by $G$ which can be simplified to
  $\left(
     \begin{array}{c}
       0 \\
       1 \\
       0 \\
     \end{array}
   \right), \left(
              \begin{array}{c}
                x \\
                0 \\
                0 \\
              \end{array}
            \right),\left(
                     \begin{array}{c}
                       0 \\
                       0 \\
                       x \\
                     \end{array}
                   \right)$ since $18y$ is a unit in $\Q(y)$. This implies that $\mathcal{N}\Q(y)[x] \cap \Z[x,y]^3=  \langle \left(
     \begin{array}{c}
       0 \\
       1 \\
       0 \\
     \end{array}
   \right), \left(
              \begin{array}{c}
                x \\
                0 \\
                0 \\
              \end{array}
            \right),\left(
                     \begin{array}{c}
                       0 \\
                       0 \\
                       x \\
                     \end{array}
                   \right)\rangle$.
\\ We can see this also directly if we compute the Gr\"{o}bner basis of $\mathcal{N}\Q(y)[x]$ over $\Q(y)[x]$:
\begin{verbatim}
ring S=(0,y),x,lp;
module N=[0,0,x*y2-x2-x*y],[0,y,x],[0,x,2x*y-x],[x,0,-x*y],[0,0,18x];
std(N);
_[1]=gen(2)
_[2]=x*gen(1)
_[3]=x*gen(3)
\end{verbatim}
If we consider the leading terms $\LT_{\Z[x]_{\gen p}[x\smallsetminus u]}$ of $G$ we obtain $18y\textbf{\textsf{e}}_2, \,\, y^3\textbf{\textsf{e}}_2, \,\, x\textbf{\textsf{e}}_1, \,\, x\textbf{\textsf{e}}_2, \,\, x\textbf{\textsf{e}}_3$, i.e. $a_1=18y, \,\, a_2=y^3, \,\, a_3=1, \,\, a_4=1, \,\, a_5=1$ and we obtain $h=\lcm(a_1, a_2, a_3, a_4, a_5)= 18y^3$. Let us compute $\mathcal{N}:18y^3$:
\begin{verbatim}
setring R;
quotient(N,18y3);
_[1]=gen(2)
_[2]=x*gen(1)
_[3]=x*gen(3)
 \end{verbatim}
We obtain again $\mathcal{N}\Q(y)[x]\cap \Z[x,y]^3$.
\\The computation shows that $\mathcal{N}=Q\cap J$ is pseudo--primary with \\$Q=\langle \left(
     \begin{array}{c}
       0 \\
       1 \\
       0 \\
     \end{array}
   \right), \left(
              \begin{array}{c}
                x \\
                0 \\
                0 \\
              \end{array}
            \right),\left(
                     \begin{array}{c}
                       0 \\
                       0 \\
                       x \\
                     \end{array}
                   \right)\rangle$ and $P=\gen x = \Ann(\mathcal{M}/Q)$. $J$ is computed in the example at the end of the paper.
\newpage
\begin{proof}of the lemma \ref{lemExtraction}
\begin{enumerate}
\item[(1)] Let $K=\sqrt{\Ann(\mathcal{M}/J)}$ and $\overline{K}=K\F_p[x]$ then\footnote{In case $p=0, \,\, \overline{K}$ is the extended ideal $K\Q[x]$. In case $p\neq 0, \,\, \overline{K}$ is the ideal induced by the canonical map $\Z[x] \longrightarrow \F_p[x]$.} $\overline{K}\supsetneq
           \overline{P}=P\F_p[x]$ since $\height(\Ann(\mathcal{M}/Q))< \height(\Ann(\mathcal{M}/J))$. This implies that $\overline{K}\cap\F_p[u]\neq \gen 0$ since
           $u \subset x$ is maximally independent for $\overline P$. Therefore $K\cap(\Z[u]\smallsetminus\langle p\rangle)\neq \gen 0$.
           Thus it holds $JA[x\smallsetminus u]=A[x\smallsetminus u]$. Finally,
           because $Q$ is primary, we obtain $\mathcal{N} A[x\smallsetminus u]\cap \mathcal{M}=
           Q A[x\smallsetminus u] \cap \mathcal{M}=Q$.

\item[(2)]  We have to prove that for every $h\in \mathcal{N}A[x\smallsetminus u]$ there exists $g\in G$ such that $\LT_{A[x\smallsetminus  u]}(g)\mid \LT_{A[x\smallsetminus  u]}(h)$.
    \\Let $h\in \mathcal{N} A[x\smallsetminus u]$. Choose $\eta\in \Z[u]\smallsetminus\langle p\rangle$ such that $\eta h\in \mathcal{N}$.
            As $h$ is a polynomial in $x \smallsetminus u$ with coefficients in $A$, the element $\eta h$ can be written
           as
           $$\eta h=q^\nu a(x\smallsetminus u)^\alpha \textbf{\textsf{e}}_i + (\text{terms in}(x\smallsetminus u)\textbf{\textsf{e}}_j
           \,\, \text{of smaller order})$$
           with $a\in \Z[u]\smallsetminus \langle p\rangle$.
            \\Since $G$ is a strong Gr\"obner basis of $\mathcal{N}\subset \mathcal{M}$ there exists a $g\in G$ such that
           $\LT_{\Z[x]}(g)\mid \LT_{\Z[x]}(\eta h)$.
           \\If $q\neq 1$ and $q^\tau$ is
           the maximal power of $q$ dividing the leading coefficient $\LC_{\Z[x]}(g)$ of $g$
           then $\tau\leq \nu$ because $\LT(g)$ divide
           $$\LT_{\Z[x]}(\eta h)=q^\nu \LT_{\Z[x]}(a)(x\smallsetminus u)^\alpha \textbf{\textsf{e}}_i.$$
           Now we can write $g$ as an element of $F_u[x\smallsetminus u]$ w.r.t.
           the corresponding ordering, i.e.
           $$g=q^\mu b(x\smallsetminus u)^\beta \textbf{\textsf{e}}_i+ (\text{terms in} (x\smallsetminus u)\textbf{\textsf{e}}_j \,\,\text{of smaller order})$$
           (by assumption $G$ is a strong Gr\"obner basis of $\mathcal{N}$ w.r.t. a block
           ordering satisfying $(x \smallsetminus u) \textbf{\textsf{e}}_i\gg u \textbf{\textsf{e}}_j$)  $b\in \Z[u]\smallsetminus \langle p\rangle$
            and $\mu\leq \tau \leq \nu$.
            \\By definition we have
            $$\LT_{A[x\smallsetminus u]}(g)=q^\mu b(x\smallsetminus u)^\beta \textbf{\textsf{e}}_i$$
            resp.
            $$\LT_{A[x\smallsetminus  u]}(h)=q^\nu \frac{a}{\eta}(x\smallsetminus u)^\alpha \textbf{\textsf{e}}_i$$
             and on the other hand it holds
           $$\LT_{\Z[x]}(g)=q^\mu \LT_{\Z[x]}(b)(x\smallsetminus u)^\beta \textbf{\textsf{e}}_i$$ resp.
           $$\LT_{\Z[x]}(\eta h)=q^\nu \LT_{\Z[x]}(a)(x\smallsetminus u)^\alpha \textbf{\textsf{e}}_i.$$
           Thus the assumption $\LT_{\Z[x]}(g) \mid \LT_{\Z[x]}(\eta h)$ implies $(x\smallsetminus u)^\beta \mid (x
           \smallsetminus u)^\alpha$ and consequently
           $\LT_{A[x\smallsetminus u]}(g) \mid \LT_{A[x \smallsetminus u]}(h)$.
\item[(3)] Obviously $(\mathcal{N}:h^\infty)\subset
           A[x\smallsetminus u]\mathcal{N}$.
           To prove the inverse inclusion let $f\in A[x\smallsetminus u]\mathcal{N}\cap \mathcal{M}$.
           This implies that $NF(f\mid G)=0$. But the normal form algorithm
           requires only to divide by the leading coefficients $\LC(g_i)\,\text{of}\,g_i\,\text{for} \,\,i=1,2,\ldots,k$.
           Hence we obtain a standard representation $f=\sum_{i=1}^kc_ig_i$ with $c_i\in \mathbb{Z}[x]_h$.
           Therefore $h^mf\in \mathcal{N}$ for some $m$. This proves $A[x\smallsetminus u]\mathcal{N}\cap \mathcal{M}\subset(\mathcal{N}:h^\infty)$.

\end{enumerate}
\end{proof}

\section{The algorithms} \label{secAlg}

In this section we present the algorithm to compute a primary decomposition
of a submodule of a free module in a polynomial ring over the integers by applying the results of section
\ref{secBasDefRes}.

\vspace{0.2cm}

The algorithm to a pseudo--primary component
is based on the Pseudo--Primary Lemma \ref{lemIntPrimaryComp}.
The algorithm to extract the primary component from the pseudo--primary component
is based on the Extraction Lemma \ref{lemExtraction}.

\begin{alg}
\textsc{modprimdecZ} \label{algPrimdecZ}
\begin{algorithmic}
\REQUIRE $F_\mathcal{N} = \{f_1, \ldots, f_k\}$, $\mathcal{N}=\langle F_{\mathcal{N}} \rangle \subseteq\Z[x]^m$.
\ENSURE  $K := \{(Q_1, P_1), \ldots, (Q_s, P_s)\}$, $\mathcal{N}=Q_1\cap \ldots \cap Q_s$
         irredundant primary decomposition with $P_i=\sqrt{Q_i}$.
\vspace{0.1cm}
\STATE $P:=\emptyset$ a list of primary decomposition
\STATE $K:=\emptyset$ a list of remaining elements
\STATE $L:= \{(\overline{Q}_1, P_1), \ldots, (\overline{Q}_r, P_r)\}:=\textsc{modpseudoprimdecZ}(\mathcal{N})$;
\FOR{$i=1, \ldots, r$}
\IF{$P_i\neq 0$}
\STATE compute $u_i$ a maximal independent set for $P_i$;
\STATE $(Q_i, h):= \textsc{modextractZ}((\overline{Q}_i,P_i),u)$;
\STATE $P:= P\cup (Q_i,P_i)$;
\STATE $K:= K\cup (\overline{Q}_i +h\Z[x]^{m})$;
\ELSE
\STATE $P:= P\cup (\overline{Q}_i, P_i)$;
\ENDIF
\ENDFOR
\FOR{$j=1, \ldots,$ size$(K)$}
\STATE $S:= \textsc{modprimdecZ}(K_j)$;
\STATE $P:= P\cup S$;
\ENDFOR
\RETURN $P$;
\end{algorithmic}
\end{alg}

\newpage
\begin{alg}
\textsc{modpseudoprimdecZ}\label{algPseudoZ}
\begin{algorithmic}
\REQUIRE $\mathcal{N}$ a submodule of the free module $\mathcal{M}$.
\ENSURE a list $R$ of pseudo--primary modules of $\mathcal{N}$ and their associated primes.
\vspace{0.1cm}
\IF{$\mathcal{N}=\mathcal{M}$}
\RETURN$\emptyset$
\ENDIF
\STATE $I:=\Ann(\mathcal{M}/\mathcal{N})$;
\IF{$\mathcal{N}=0$}
\RETURN$(\mathcal{N},0)$;
\ELSE
\STATE compute $B:=\{P_1, \ldots, P_r\}$, the set of minimal associated primes of $I$;
\STATE comute $\{s_1, \ldots, s_r\}$ a system of separators for $B$;
\FOR{$i=1, \ldots, r$}
\STATE compute the saturation $Q_i$ of $N$ w.r.t $s_i$ and the integer $k_i$, the index of the saturation.
\STATE $R:=\{(Q_1,P_1), \ldots, (Q_r,P_r)\}$;
\STATE $L=\textsc{modpseudoprimdecZ}(\mathcal{N}+\langle s_1^{k_1}, \ldots, s_r^{k_r}\rangle\Z[x]^m)$;
\ENDFOR
\ENDIF
\RETURN $R\cup L$;
\end{algorithmic}
\end{alg}
\begin{alg}
\textsc{modextractZ} \label{algExtractZ}
\begin{algorithmic}
\REQUIRE $K$ the list of a pseudo--primary module the corresponding minimal associated prime and
         $L$ is a list of maximal independent set $u$ for the prime ideal.
\ENSURE  The primary component $Q$ of $\mathcal{N}$ associated to $P$ and a polynomial $h$.
\vspace{0.1cm}
\STATE $I:= \Ann(\mathcal{M}/Q)$;
\STATE compute $G=\{g_1,\ldots,g_k\}$, a strong Gr\"obner basis of $\mathcal{N}$ w.r.t. a block
       ordering satisfying $x \smallsetminus u \gg u$;
\IF{$I\cap \Z=0$}
\STATE compute $\{a_1,\ldots,a_k\}$ such that
       $LC_{\Z(u)[x \smallsetminus u]}(g_i)= a_i$ with
       $a_i \in \Z[u]$;
\STATE compute $h=lcm(a_1,\ldots,a_k)$, the least common multiple of $a_1,\ldots,a_k$;
\ENDIF
\IF{$I\cap \Z=\langle p\rangle$, $p\neq 0$}
\STATE compute $\{a_1,\ldots,a_k\}$ such that
       $LC_{\Z[u]_{\langle p \rangle}[x \smallsetminus u]}(g_i)= p^{\nu_i} \cdot a_i$ with
       $a_i \in \Z[u] \smallsetminus \gen p$;
\STATE compute $h=lcm(a_1,\ldots,a_k)$, the least common multiple of $a_1,\ldots,a_k$;
\ENDIF
\STATE compute the saturation $Q$ of $\mathcal{N}$ w.r.t $h$ and $k$, the index of saturation.
\RETURN $(Q, h^k)$;
\end{algorithmic}
\end{alg}

\newpage

\section{Example} \label{secExTime}

We have implemented the algorithms (cf. section \ref{secAlg}) in \singular in the library \texttt{primdecint.lib} (cf. \cite{DGPS})

\begin{exmp} \label{ex1}
\begin{verbatim}

LIB"primdecint.lib";
ring R=integer,(x,y),(c,lp);
module N=[0,0,xy2-x2-xy],[0,y,x],[0,x,2xy-x],[x,0,-xy],[0,0,18x];
> pseudo_primdecZM(N);
[1]:
   [1]:
      _[1]=[0,0,18x]
      _[2]=[0,0,xy2]
      _[3]=[0,0,x2-2xy2+xy]
      _[4]=[0,y,x]
      _[5]=[0,x,2xy-x]
      _[6]=[x,0,-xy]
   [2]:
      _[1]=x
\end{verbatim}

\begin{verbatim}
> primdecZM(N);
[1]:
   [1]:
      _[1]=[0,0,x]
      _[2]=[0,1]
      _[3]=[x,0,-xy]
   [2]:
      _[1]=x
[2]:
   [1]:
      _[1]=[0,0,y3]
      _[2]=[0,0,18x]
      _[3]=[0,0,xy2]
      _[4]=[0,0,x2-2xy2+xy]
      _[5]=[0,y,x]
      _[6]=[0,x,2xy-x]
      _[7]=[y3]
      _[8]=[x,0,-xy]
   [2]:
      _[1]=y
      _[2]=x
      \end{verbatim}
The computation shows that the module $\mathcal{N}$ is pseudo--primary with minimal associated prime $\langle x\rangle$ it has an embedded component with associated prime $\langle x, y\rangle$.
      \end{exmp}
\section{Acknowledgement}
The authors would like to thank the reviewer for all the useful hints.


\begin{thebibliography}{99}
\bibitem {WP} Adams, W.W.; Loustaunau, P.: An Introduction to Gr\"{o}bner bases. Graduate studies in mathematics, vol. 3, American Mathematical Scociety, 2003.
\bibitem{A} Ayoub, C.W.: The Decomposition Theorem for Ideals in Polynomial Rings over
  a Domain. Journal of Algebra 76, 99--110 (1982).
\bibitem{DGP} Decker, W.; Greuel, G.-M.; Pfister, G.: Primary Decomposition:
  Algorithms and Comparisons. In: Algorithmic Algebra and Number Theory, Springer, 187--220
  (1998).
\bibitem{DGPS} Decker, W.; Greuel, G.-M.; Pfister, G.; Sch{\"o}nemann, H.:
  \newblock {\sc Singular} {3-1-6} --- {A} computer algebra system for polynomial
  computations. \newblock {http://www.singular.uni-kl.de} (2013).
\bibitem{EHV} Eisenbud, D.; Huneke, C.; Vasconcelos, W.: Direct Methods for Primary
  Decomposition. Inventiones Mathematicae 110, 207--235 (1992).
\bibitem{GP} Greuel, G.-M.; Pfister, G.: A \textsc{Singular} Introduction to
  Commutative Algebra. Second edition, Springer (2007).
\bibitem{GTZ} Gianni, P.; Trager, B.; Zacharias, G.: Gr\"obner Bases and Primary
  Decomposition of Polynomial Ideals. Journal of Symbolic Computation 6, 149--167 (1988).
\bibitem{I} Idrees, N.: Algorithms for primary decoposition of modules. Studia Scientiarum Mathematicarum Hungarica 48 (2), 227-246 (2011).
\bibitem{PSS} Pfister, G.; Sadiq, A.; Steidel, S.: An Algorithm for Primary Decomposition in Polynomial Rings over the Integers. Central European Journal of Mathematics Vol. 9, No. 4, (2010) 897-904
\bibitem{R} Rutman, E.W.: Gr\"{o}bner bases and primary decomposition of modules. J. Symbolic Computation (1992)14, 483-503.
\bibitem{Sa} Sadiq, A.: Standard bases over Rings. International Journal of Algebra and Computation, Vol. 20, No. 7 (2010) 953-968.
\bibitem{Se} Seidenberg, A.: Constructions in a Polynomial Ring over the Ring of
  Integers. American Journal of Mathematics 100 (No. 4), 685--703 (1978).
\bibitem{SY} Shimoyama, T.; Yokoyama, K.: Localization and Primary Decomposition of
  Polynomial Ideals. Journal of Symbolic Computation 22, 247--277 (1996).

\end{thebibliography}
\end{document}